\documentclass[12pt]{article}
\usepackage{tabularx} 
\usepackage{amsmath}
\usepackage{amsthm} 
\usepackage{graphicx} 
\usepackage[margin=1in,letterpaper]{geometry} 
\usepackage{cite} 
\usepackage[final]{hyperref} 
\hypersetup{
	colorlinks=true,       
	linkcolor=blue,        
	citecolor=blue,        
	filecolor=magenta,     
	urlcolor=blue
}
\usepackage{cite}
\usepackage{blindtext}
\numberwithin{equation}{section}
\usepackage{amssymb}
\usepackage{mathrsfs}
\usepackage{amsmath}
\newtheorem{theorem}{Theorem}[section]
\newtheorem{claim}{Claim}
\newtheorem{corollary}{Corollary}[section]

\newtheorem{lemma}{Lemma}[section]
\newtheorem{Remark}{Remark}[section]

\begin{document}
\bibliographystyle{plain}
\title{Linear maps preserving $(p,k)$ norms of tensor products of matrices}
\author{Zejun Huang\thanks{ College of Mathematics and Statistics, Shenzhen University, Shenzhen 518060, China.  (Email: mathzejun@gmail.com) } , Nung-Sing Sze\thanks{
Department of Applied Mathematics,
The Hong Kong Polytechnic University,
Hung Hom, Hong Kong.
(Email: raymond.sze@polyu.edu.hk)}
, Run Zheng\thanks{
Department of Applied Mathematics,
The Hong Kong Polytechnic University,
Hung Hom, Hong Kong.
(Email: zheng-run.zheng@connect.polyu.hk )}
}  
\date{ }
\maketitle

\begin{abstract}
Let $m,n\ge 2$ be integers. Denote by $M_n$ the set of $n\times n$ complex matrices.
Let $\|\cdot\|_{(p,k)}$ be the $(p,k)$ norm on $M_{mn}$ with $1\leq k\leq mn$ and $2<p<\infty$.  We show that a  linear map $\phi:M_{mn}\rightarrow M_{mn}$ satisfies
$$\|\phi(A\otimes B)\|_{(p,k)}=\|A\otimes B\|_{(p,k)} {\rm\quad for~ all\quad}A\in M_m {\rm ~and ~}B\in M_n$$   if and only if there exist unitary matrices $U,V\in M_{mn}$ such that 
$$\phi(A\otimes B)=U(\varphi_1(A)\otimes  \varphi_2(B))V  {\rm\quad for~ all\quad}A\in M_m {\rm~ and~ }B\in M_n,$$ where $\varphi_s$ is the identity map or the transposition map  $X\to X^T$ for $s=1,2$.  The result is also extended to multipartite systems.
\end{abstract}


\section{Introduction}
Throughout this paper, we denote by $M_{m,n}$ and $M_n$ the set  of  $m\times n$ and $n\times n$ complex matrices, respectively. Denote by $H_n$ the set of all $n\times n$ Hermitian  matrices. For two matrices $A=(a_{ij})\in M_m$ and $B\in M_n$, their tensor product is defined to be $A\otimes B=(a_{ij}B)$, which is an $mn\times mn$ matrix. We denote by $$M_m\otimes M_n=\{A\otimes B: A\in M_m, B\in M_n\}$$ and
$$ H_m\otimes H_n=\{A\otimes B: A\in H_m, B\in H_n\}.$$

 Suppose $A\in M_{m,n}$. The singular values of $A$ are always denoted in decreasing order by
 $s_1(A)\geq \cdots\geq s_{\ell}(A)$, where $\ell=\min\{m,n\}$. Given a real number  $p\geq 1$ and a positive integer  $k\leq \min\{m,n\}$, the $(p,k)$-norm of $A$ is defined by
$$\|A\|_{(p,k)}=\left[\sum\limits_{i=1}^ks_i^p(A)\right]^{
\frac{1}{p}}.$$ 
{The $(p,k)$-norm, also known as the Ky Fan (p,k)-norm,  was first recongnized as a special class of unitarily invariant norms in the study of isometries by Grone and Marcus \cite{GM1977} in their notable work from the 1970s.
The $(p,k)$-norms encompass many commonly used norms.  
For instance, the $(1,k)$-norm reduces to the Ky Fan $k$-norm 
while the $(p,K)$-norm, with $K = \min\{m,n\}$, reduces to Schatten $p$-norm.
Moreover, the Ky Fan $1$-norm, Ky Fan $K$-nrom, and Shatten $2$-norm are also known as  the spectral norm., the trace norm, and the Frobenius norm, respectively. 
Some earlier works exploring the fundamental properties of the $(p,k)$-norm can be found in references \cite{JC1987, kidman1983,MKS1984}.
}

{
In addition to being a generalation of many well-known norms, 
the $(p,k)$-norm itself has attracted extensive attention from researchers across various fields,  particularly in 
the study of low-rank approximation, e.g., \cite{DV2016, JLQD, TN2014}.
The application of the $(p,k)$-norm in quantum information science has also gained recent attention. 
Researchers in this field have explored the concept of the twisted commutators of two unitaries and focused on determining the minimum norm value of these twisted commutators. 
The authors in \cite{CF2017} succeeded in obtaining an explicit closed form for the minimum twisted commutation value with respect to the $(p,k)$-norm. 
All these show the growing importance and relevance of the $(p,k)$-norm across various fields of study.
}

 Linear preserver problems concern the study of linear maps on matrices or operators preserving certain special properties. Since Frobenius gave the characterization of linear maps on $M_n$ that preserve the determinant of all matrices in 1897, a lot of  linear preserver problems have been investigated; see \cite{LP, LM} and their references.

{
The study of linear preservers on various matrix norms  have been extensively explored since Schur \cite{S1925} characterized linear maps on $M_n$ that preserve the spectral norm.
This was followed by a series of subsequent results
\cite{A1975,G1980,GM1977,LT1988,R1969,R19692}.
Notably,  Li and Tsing \cite{LT1988} 
provided a complete characterization of linear maps that preserve the $(p,k)$-norms.
They showed that linear maps on $M_{m,n}$ that preserve the $(p,k)$-norms (except for the Frobenius norm) have the form
$$A\mapsto UAV \quad\hbox{or}\quad A\mapsto UA^TV   \text{\quad when\quad} m=n $$
for some unitary matrices $U\in M_m$ and $V\in M_n.$
}

Traditional linear preserver problems deal with linear 
maps  preserving certain properties of every matrix in 
the whole matrix space  $M_n$ or $H_n$. 
Recently, linear maps on $M_{mn}$  or $H_{mn}$  only preserving certain properties of matrices in $M_m\otimes M_n$  or $H_m\otimes H_n$ have been investigated. 
{Friedland et al. \cite{FLTS2011} provided a characterization of linear maps on $H_m\otimes H_n$ that preserve the set of separable states in bipartite systems. The concept of separability is widely recognized as a fundamental and crucial aspect in the field of quantum information science.  
Johnston in his paper \cite{J2011} examined invertible linear maps on $M_m\otimes M_n$ that preserve the set of rank one matrices with bounded Schmidt rank in both row and column spaces.
Additionally, the author investigated linear maps on $M_m\otimes M_n$ that preserve
the Schmidt $k$-norm,  a norm induced by states with bounded Schmidt rank,
which finds extensive application in the field of quantum information.
For more details on the the Schmidt $k$-norm, refer to \cite{JK2010}.}

Note that $M_m\otimes M_n$ and  $H_m\otimes H_n$ are small subsets of $M_{mn}$ and $H_{mn}$.  Researchers know much less information on such linear maps.  So it is more difficult to characterize such linear maps. Along this line, 
linear maps on Hermitian matrices preserving the spectral radius  were determined in \cite{FHLS2}. {Linear maps on complex matrices or Hermitian matrices preserving  determinant were studied in \cite{CK2020, DC2016, DF2017}}.  Linear maps on complex matrices preserving numerical radius, $k$-numerical range, product numerical range and rank-one matrices were characterized in \cite{FHLS1,FHLS3, HSS, LTS2013}. 

In \cite{FHLS4}, the authors characterized linear maps on $M_{mn}$ preserving the Ky Fan $k$-norm   and the Schatten $p$-norm of the tensor products $A\otimes B$ for all  $A\in M_m$ and $B\in M_n$. 
{Despite the non-obvious connection to the field of quantum information, from a mathematical perspective,  it is undeniably intriguing to consider the linear maps that preserve the $(p,k)$-norm of tensor products of matrices.} 
Therefore, in this paper, we aim to characterize linear maps $\phi$ on $M_{mn}$ such that for $p>2$ and $1\leq  k\leq mn$,
\begin{equation}\label{h717}
\|\phi(A\otimes B)\|_{(p,k)}=\|A\otimes B\|_{(p,k)}\quad \hbox{for all\quad}A\in M_m\hbox{ and }B\in M_n.
\end{equation}
{The comprehensive characterization in the bipartite systems will be presented in Section 2, while in Section 3, we will extend the results to multipartite systems.}

\section{Bipartite system}
 The linear maps on $M_{mn}$  satisfying (\ref{h717}) are determined by the following theorem.
\begin{theorem}\label{CH3theo1}
Let $m,n \geq 2 $ be integers.   Given a real number $p>2$ and a positive  integer $ k \leq mn$,
a linear map $\phi: M_{mn}\to M_{mn}$ satisfies
\begin{equation}\label{CH3t1}
\|\phi(A\otimes B)\|_{(p,k)}=\|A\otimes B\|_{(p,k)} \quad \hbox{for all\quad $A\in M_m$ and $ B\in M_n$,}
\end{equation}
if and only if there exist unitary matrices $U,V \in M_{mn}$ such that
\begin{equation}\label{CH3t2}
\phi(A\otimes B)=U(\varphi_1(A)\otimes \varphi_2(B))V \quad \hbox{for all\quad  $A\in M_m$ and $ B\in M_n$,}
\end{equation}
where $\varphi_s$ is the identity map or  the transposition map $X\mapsto X^T$ for $s=1,2$.
\end{theorem}

To prove the theorem, we need some notations and preliminary results.
{Denote by $\|A\|$ and $A^*$ the Frobenius norm and  the conjugate transpose of the matrix $A$, respectively.}  Two matrices $A,B\in M_n$ are said to be orthogonal, denoted by $A\perp B$, if $A^*B=AB^*=0$. Denote by $E_{ij}\in M_{m,n}$ the matrix whose $(i,j)$-th entry is equal to one and all the other entries are equal to zero.

The eigenvalues of an $n\times n$ Hermitian  matrix $A$ are always denoted in decreasing order by $\lambda_1(A)\geq \lambda_2(A)\geq\cdots\geq  \lambda_n(A)$. For   $A,B\in H_n$, we use the notation $A\geq B$  or $B\leq A$  to mean that $A-B$ is positive semidefinite. Let $\mathbb{R}$ be the set of all real numbers.
Rearrange the components of $x=(x_1,\ldots,x_n)\in \mathbb{R}^n$ in  decreasing order as $x_{[1]}\geq \cdots \geq x_{[n]}$. For  $x=(x_1,\ldots,x_n),$
 $y=(y_1,\ldots,y_n)\in \mathbb{R}^n$, if $$\sum_{i=1}^kx_{[i]}\leq
 \sum_{i=1}^ky_{[i]} \hbox{\quad for\quad}k=1,\ldots,n,$$
 then we say
 $x$ is  {\it weakly majorized} by $y$ and denote by $y\succ_w x$ or $x\prec_w y$.

  Notice that $x\mapsto x^{\gamma}$ $(x\geq 0)$ is a convex function for any real number $ \gamma\geq 1$.   One can easily conclude the following lemma.
\begin{lemma}\label{CH3lemma2}

Let $a,b\in \mathbb{R}$. If  $-a\leq b\leq a$, then for any real number $\gamma\geq 1$, 
\begin{equation}\label{le2.2}
(a+b)^{\gamma}+(a-b)^{\gamma}\geq 2a^{\gamma}.
\end{equation}
\end{lemma}
The following lemmas are crucial in our proof.
\begin{lemma}\textnormal{\cite[Lemma 2.1]{M1967}}
\label{CH3lemma1}
Let $A\in M_n$ be a positive semidefinite matrix.  Then
$$x^*A^{\gamma}x\geq (x^*Ax)^{\gamma}\|x\|^{2(1-\gamma)}\quad  \hbox{for all\quad $x\in \mathbb{C}^n$ and $\gamma\geq 1 $}. $$
\end{lemma}

\begin{lemma}\label{CH3lemma22}\textnormal{\cite[Lemma 3.7]{ZHAN2013}}
 Let $A\in M_n$ be a Hermitian matrix and $k \leq n$ be a positive integer. Then
\begin{equation}\notag
\sum_{i=1}^k\lambda_i(A)=\max\limits_{U^*U=I_k}\mathrm{tr}(U^*AU)\quad \hbox{and}\quad
\sum_{i=1}^k\lambda_{n-i+1}(A)=\min\limits_{U^*U=I_k}\mathrm{tr}(U^*AU),
\end{equation}
where  $I_k$ is the identity matrix of order $k$
and $U\in M_{n,k}$.
\end{lemma}

\begin{lemma}\textnormal{\cite[Lemma 1]{LS2003}}\label{lemma4}
Let $A, B\in  M_n$. Then  $A\perp B$ if and only if there exist $\hat{A}\in M_m $, $\hat{B}\in M_{n-m}$ and  unitary matrices $U,V\in M_n$  such that
$$UAV=\hat{A}\oplus 0\quad \hbox{and}\quad UBV=0\oplus \hat{B}.$$
\end{lemma}

\begin{lemma}\label{adle2}
Let $A, B,C\in M_{n}$.   If $(A+B)\perp C$ and  $A\perp B$,  then
$$A\perp C\quad \hbox{and} \quad B\perp C.$$
\end{lemma}
\begin{proof}
Since $A\perp B$, we can apply Lemma \ref{lemma4} to conclude that there exist $\hat{A}\in M_m $, $\hat{B}\in M_{n-m}$ and  unitary matrices $U,V\in M_n$  such that
$$UAV=\hat{A}\oplus 0\quad \hbox{and}\quad UBV=0\oplus \hat{B}.$$
Let $UCV$ be partitioned as
$$UCV=\begin{bmatrix}
C_{11}&C_{12}\\
C_{21}&C_{22}
\end{bmatrix}
$$ with $C_{11}\in M_m$ and $C_{22}\in M_{n-m}.$ It follows from $(A+B)\perp C$ that
\begin{equation}\notag
(U(A+B)V)^{*}(UCV)=0 \quad \hbox{and}\quad (U(A+B)V)(UCV)^*=0,
\end{equation}
that is,
$$\begin{bmatrix}
\hat{A}^*C_{11}&\hat{A}^*C_{12}\\
\hat{B}^{*}C_{21}&\hat{B}^{*}C_{22}
\end{bmatrix}=0\quad \hbox{and}\quad
\begin{bmatrix}
\hat{A}C_{11}^*&\hat{A}C_{21}^*\\
\hat{B}C_{12}^*&\hat{B}C_{22}^*
\end{bmatrix}=0.
$$
Then we have 
$$V^*A^*CV=(UAV)^*(UCV)=
\begin{bmatrix}
\hat{A}^*C_{11}&\hat{A}^*C_{12}\\
0&0
\end{bmatrix}=0
$$
and
$$UAC^*U^*=(UAV)(UCV)^*=\begin{bmatrix}
\hat{A}C_{11}^*&\hat{A}C_{21}^*\\
0&0
\end{bmatrix}=0
.$$
Thus,  $A^{*}C=0$ and $AC^{*}=0$, i.e., $A\perp C$. Similarly, we can also conclude that $B\perp C$.
\end{proof}

\begin{lemma}\label{CH3lemma3}
Let $C,D\in M_n$ be two Hermitian matrices such that $-C\leq D\leq C$ and $ k\leq n$  be a positive integer. Then for any real number $ \gamma\geq 1$,
\begin{equation}\label{le2.7}
\sum_{i=1}^{k}\lambda_i^{\gamma}(C+D)+\sum_{i=1}^{k}\lambda_i^{\gamma}(C-D)\geq 2\sum_{i=1}^k\lambda_i^{\gamma}(C).
\end{equation}
\end{lemma}

\begin{proof}
Let $U\in M_n$ be a unitary matrix  such that  $$U^*CU=\mathrm{diag}(\lambda_1(C),\lambda_2(C),\ldots,\lambda_n({C})).$$ Denote by $u_i$ the $i$-th  column of $U$ for $i=1,\ldots,n$.  Let $\hat{U}=[u_1,u_2\ldots,u_k].$ Then applying Lemma \ref{CH3lemma22}, we have
\begin{equation}\notag
\sum_{i=1}^k\lambda_{i}^{\gamma}(C+D)=\sum_{i=1}^k\lambda_{i}\big((C+D)^{\gamma}\big)\geq \mathrm{tr}(\hat{U}^*(C+D)^{\gamma}\hat{U})
\end{equation}  and
$$\quad \sum_{i=1}^k\lambda_{i}^{\gamma}(C-D)=\sum_{i=1}^k\lambda_{i}\big((C-D)^{\gamma}\big)\geq \mathrm{tr}(\hat{U}^*(C-D)^{\gamma}\hat{U}).$$
Since $-C\leq D\leq C$,  we have
 $$-x^*Cx\leq x^*Dx\leq x^*Cx \quad\hbox{ for all\quad }  x\in \mathbb{C}^{n}.$$
By Lemma \ref{CH3lemma1},  we have  
\begin{equation}\notag
u_i^{*}(C+D)^{\gamma}u_{i}\geq (u_i^*(C+D)u_i)^{\gamma} \quad \hbox{and} \quad u_i^{*}(C-D)^{\gamma}u_{i}\geq (u_i^*(C-D)u_i)^{\gamma}
\end{equation}
for  $i=1,\ldots,n$.   Applying Lemma \ref{CH3lemma2} with $a=u_i^{*}Cu_i$ and $b=u_i^{*}Du_i$, we get
\begin{equation}\notag
(u_i^{*}(C+D)u_i)^{\gamma}+(u_i^{*}(C-D)u_i)^{\gamma}\geq 2(u_i^{*}Cu_i)^{\gamma}\quad \hbox{for\quad $i=1,\ldots, n.$}
\end{equation}
It follows from the above inequalities  that
\begin{equation}\notag
\begin{split}\sum_{i=1}^k\lambda_{i}^{\gamma}(C+D)+\sum_{i=1}^k\lambda_{i}^{\gamma}(C-D)\geq& \mathrm{tr}(\hat{U}^*(C+D)^{\gamma}\hat{U})+\mathrm{tr}(\hat{U}^*(C-D)^{\gamma}\hat{U})\\
=&\sum_{i=1}^{k}u_i^{*}(C+D)^{\gamma}u_i+\sum_{i=1}^{k}u_i^{*}(C-D)^{\gamma}u_i
\\
\geq&\sum_{i=1}^{k}(u_i^{*}(C+D)u_i)^{\gamma}+\sum_{i=1}^{k}(u_i^{*}(C-D)u_i)^{\gamma} \\
\geq& 2 \sum_{i=1}^{k}(u_i^{*}Cu_i)^{\gamma}=2\sum_{i=1}^k\lambda_i^{\gamma}(C).
\\
\end{split}
\end{equation}\end{proof}

{\begin{Remark}\label{remark}
The inequality (\ref{le2.7}) can be regarded as a generalization of the inequality (\ref{le2.2}) in Lemma \ref{CH3lemma2}.  It is worth noting that if $-a\leq b\leq a$, then  
$$(a+b)^{\gamma}+(a-b)^{\gamma}\leq 2a^{\gamma}\quad \hbox{for all}\quad  0<\gamma<1.$$
In our attempt to generalize this inequality,  we aimed to obtain the following analogous inequality to (\ref{le2.7}):
$$\sum_{i=1}^{k}\lambda_i^{\gamma}(C+D)+\sum_{i=1}^{k}\lambda_i^{\gamma}(C-D)\leq  2\sum_{i=1}^k\lambda_i^{\gamma}(C)\quad \hbox{for all }\quad 0<\gamma<1,$$ 
where $1\leq k\leq n$ and  $C,D\in M_n$ are Hermitian matrices such that $C+D$ and $C-D$ are both positive semidefinite. However, it has been demonstrated that this inequality does not hold in general. 
A counterexample can be constructed by considering matrices $C$ and $D$ such that $C+D=\mathrm{diag}(1,1,3,3)$ and $C-D=\mathrm{diag}(3,3,1,1)$. In this case,  we observe that
$\sum\limits_{i=1}^{2}\lambda_i^{\gamma}(C+D)+\sum\limits_{i=1}^{2}\lambda_i^{\gamma}(C-D)= 4\cdot 3^{\gamma}>2\sum\limits_{i=1}^2\lambda_i^{\gamma}(C)=4\cdot 2^{\gamma}.$
\end{Remark}
}

\begin{corollary}\label{CH3c1}
Let  $p>2$ be a real number and $ k\leq n$ be a positive  integer.   Then
\begin{equation}\label{CH3eqqq1}
\|A+B\|_{(p,k)}^p+\|A-B\|_{(p,k)}^p\geq 2\sum_{i=1}^k\lambda_i^{\frac{p}{2}}(A^*A+B^*B)
\end{equation}
for all $A,B\in M_n.$
\end{corollary}
\begin{proof}
Notice that $$\|A+B\|_{(p,k)}^p=\sum_{i=1}^ks_i^p(A+B)=\sum_{i=1}^k\lambda_i^{\frac{p}{2}}\big((A^*A+B^*B)+(A^*B+B^*A)\big)$$ and $$\|A-B\|_{(p,k)}^p=\sum_{i=1}^ks_i^p(A-B)=\sum_{i=1}^k\lambda_i^{\frac{p}{2}}\big((A^*A+B^*B)-(A^*B+B^*A)\big).$$
Let  $C=A^*A+B^*B$ and $D= A^*B+B^*A$. Then $C+D=(A+B)^*(A+B)$ and $C-D=(A-B)^*(A-B)$ are both positive semidefinite, that is, $-C\leq D\leq C$.
Applying Lemma \ref{CH3lemma3},
we get (\ref{CH3eqqq1}).\end{proof}

\begin{lemma}\label{CH3le2}
Let $A,B\in M_n$ be nonzero matrices and $ 2\leq k\leq n$ be an integer. Given a real number $p\geq 1,$ if
$$\|A+B\|_{(p,k)}^p=\|A\|_{(p,k)}^p+\|B\|_{(p,k)}^p\quad \hbox{and}\quad A\perp B,$$ then $\mathrm{rank}(A+B)\leq k.$
\end{lemma}
\begin{proof}
With the assumption that $A\perp B$, we can
assume that  the largest $k$ singular values of $A+B$ are $s_1(A),\ldots,s_{\ell}(A),s_1(B),\ldots,s_{k-\ell}(B)$ for some $0\leq \ell \leq k$.
Then\begin{equation}\label{CH332}
\|A+B\|_{(p,k)}^p=\sum_{i=1}^{\ell}s_i^p(A)+\sum_{i=1}^{k-\ell}s_i^p(B)\leq \sum_{i=1}^{k}s_i^p(A)+\sum_{i=1}^{k}s_i^p(B).
\end{equation}
On the other hand, we have
$$\|A+B\|_{(p,k)}^p=\|A\|_{(p,k)}^p+\|B\|_{(p,k)}^p=\displaystyle\sum_{i=1}^{k}s_i^p(A)+\displaystyle\sum_{i=1}^{k}s_i^p(B). $$ Thus, the equality in (\ref{CH332})  holds, which implies
\begin{equation}\label{addd}
\sum_{i=1}^{\ell}s_i^p(A)=\sum_{i=1}^{k}s_i^p(A)\quad \hbox{and}\quad \sum_{i=1}^{k-\ell}s_i^p(B)=\sum_{i=1}^{k}s_i^p(B).
\end{equation}
Since $A$ and $B$ are both nonzero, we have $$\displaystyle\sum_{i=1}^{k}s_i^p(A)>0\quad {\rm and}\quad   \displaystyle\sum_{i=1}^{k}s_i^p(B)>0,$$
 which implies $\ell \geq 1$ and $k-\ell\geq 1$, i.e., $1\leq \ell\leq k-1.$ With (\ref{addd}), it follows that
  \begin{equation}\notag
\sum_{i=\ell+1}^{k}s_i^p(A)=0\quad \hbox{and} \quad \displaystyle \sum_{i=k-\ell+1}^{k}s_k^p(B)=0,
\end{equation}
which implies    $s_{\ell+1}(A)=0$ and $s_{k-\ell+1}(B)=0.$ Therefore,  $$\mathrm{rank}(A)\leq \ell \quad{\rm and } \quad\mathrm{rank}(B)\leq k-\ell.$$ Since $A\perp B$, we have $$\mathrm{rank}(A+B)=\mathrm{rank}(A)+\mathrm{rank}(B)\leq \ell+k-\ell=k.$$ \end{proof}

\begin{lemma}\label{lemma8}
Let $A, B\in M_n$ be  two positive semidefinite matrices,  $\gamma>1$ be a real number and  $k\leq n$ be a positive  integer. Suppose
 \begin{equation}\label{CH3AB}
\sum_{i=1}^{k}\lambda_{i}^{\gamma}(A+\alpha B)\leq \sum_{i=1}^{k}\lambda_i^{\gamma}(A)+\sum_{i=1}^{k}\lambda_i^{\gamma}(\alpha B) \quad \hbox{for all\quad $0<\alpha<1$ }
\end{equation} and $U^*AU=\mathrm{diag}(\lambda_1(A),\ldots,\lambda_n(A))$ for some unitary  matrix $U\in M_n$.
 \item[$\mathrm{ (a)}$]If  $\lambda_k(A)=0,$  then $A\perp B$.
\item[$\mathrm{ (b)}$]
If $\lambda_k(A)>0,$ then $U^*BU=0_{k+\ell}\oplus \hat{B}$
with $\hat{B}\in M_{n-k-\ell}$, where $\ell$ is the largest integer such that $\lambda_{k+\ell}(A)=\lambda_k(A)$.
\end{lemma}

\begin{proof}
  Denote  the $i$-th diagonal entry of  $U^*BU$ by $b_i.$
 Then $\lambda_i(A)+\alpha b_{i}$ is  the $i$-th  diagonal entry of $U^*(A+\alpha B)U$.
It follows that   $$
\big(\lambda_1(A+\alpha B),\ldots,\lambda_k(A+\alpha B)
 \big )\succ_w \big (\lambda_{1}(A)+\alpha b_{1},\ldots,\lambda_{k}(A)+\alpha b_{k}\big).
$$
Notice that  $g(x)=x^{\gamma}$ $(x>0)$ is an increasing convex function when  $\gamma>1$. We can apply the  Theorem 3.26 in \cite{ZHAN2013} to obtain  $$\big(\lambda_1^{\gamma}(A+\alpha B),\ldots,\lambda_k^{\gamma}(A+\alpha B)
 \big)\succ_w \big ((\lambda_{1}(A)+\alpha b_{1})^{\gamma},\ldots,(\lambda_{k}(A)+\alpha b_{k})^{\gamma}\big).$$
 Thus, $\displaystyle\sum_{i=1}^{k}\lambda_{i}^{\gamma}(A+\alpha B)\geq \displaystyle\sum_{i=1}^{k}(\lambda_i(A)+\alpha b_i)^{\gamma}.$
 With the assumption in (\ref{CH3AB}), we can conclude  that
\begin{equation}\label{CH3eq1}
 \sum_{i=1}^{k}(\lambda_i(A)+\alpha b_{i})^{\gamma} \leq \sum_{i=1}^{k}\lambda_i^{\gamma}(A)+\sum_{i=1}^{k}\lambda_i^{\gamma}(\alpha B)\quad \hbox{for all\quad $0<\alpha<1$}.
 \end{equation}
  Let $f(\alpha)=\displaystyle\sum_{i=1}^{k}(\lambda_i(A)+\alpha b_{i})^{\gamma} -\displaystyle\sum_{i=1}^{k}\lambda_i^{\gamma}(A)-\displaystyle\sum_{i=1}^{k}\lambda_i^{\gamma}(\alpha B)$ be a function on $\alpha$.
 Then  we have
 \begin{equation}\label{CH3s1}
 f(\alpha)=f(0)+f^{'}(0)\alpha+o(\alpha)=
 \left[\sum_{i=1}^{k}\lambda_i^{\gamma-1}(A)b_{i}\gamma\right]\alpha+o(\alpha),
 \end{equation}
where a function $g(\alpha)=o(\alpha)$ means $ \lim\limits_{\alpha\to0}\frac{g(\alpha)}{\alpha}=0$.
Since $A$ and $B$ are both positive semidefinite, we have $\lambda_i(A)\geq 0$ and $ b_{i}\geq 0$ for all $i=1,\ldots,n.$ It follows that
$\displaystyle\sum_{i=1}^{k}\lambda_i^{\gamma-1}(A)b_{i}\gamma \geq 0.$
We claim that $\displaystyle\sum_{i=1}^{k}\lambda_i^{\gamma-1}(A)b_{i}\gamma=0$.  Otherwise, $\displaystyle\sum_{i=1}^{k}\lambda_i^{\gamma-1}(A)b_{i}\gamma>0$ leads to $f(\alpha)>0$ when $\alpha>0$ is sufficiently small, which contradicts (\ref{CH3eq1}).
It follows that $$\lambda_i(A)b_{i}=0 \quad  \hbox{for\quad} i=1,\ldots,k.$$

 For the case $\lambda_k(A)=0$, we may assume that $ t$ is the  largest integer  such that
 $\lambda_t(A)>0$. Then $U^*AU=\mathrm{diag}(\lambda_1(A),\ldots,\lambda_t(A))\oplus 0_{n-t}$ and  $b_i=0$  for  $i=1,\ldots,t.$
 Recall that $B$ is positive semidefinite. Thus, $U^*BU=0_{t}\oplus \hat{B}$ with $\hat{B}\in M_{n-t}.$ It follows  that $A\perp B.$

  For the case $\lambda_k(A)>0,$
  we first  have $b_i=0$ for all $i=1,\ldots,k$. Since $B$ is positive semidefinite, it follows that $U^*BU=0_k\oplus C$ with $C\in M_{n-k}.$ Recall that $\ell$ is the largest integer such that $\lambda_{k+\ell}(A)=\lambda_k(A)$.
  If $\ell=0$, then the proof is completed.  If $\ell>0$, then for any $i=k+1,\ldots,k+\ell$, replacing the role of $\lambda_k(A)+\alpha b_k$ with $\lambda_i(A)+\alpha b_i$ in the above argument, we can conclude $b_i=0$. Thus, we have $b_i=0$ for $i=1,\ldots,k+\ell.$ It follows that  $U^*BU=0_{k+\ell}\oplus \hat{B}$ with  $\hat{B}\in M_{n-k-\ell}$. \end{proof}
\begin{corollary}\label{c2}
Let $T,S\in  M_n$ be two matrices,  $p>2$ be a real number and  $k\leq n$ be a positive  integer. Suppose
$$\sum_{i=1}^k\lambda_i^{\frac{p}{2}}(T^*T+x^2S^*S)\leq \sum_{i=1}^k\lambda_i^{\frac{p}{2}}(T^*T)+\sum_{i=1}^k\lambda_i^{\frac{p}{2}}(x^2S^*S)$$
and
$$\sum_{i=1}^k\lambda_i^{\frac{p}{2}}(TT^*+x^2SS^*)\leq \sum_{i=1}^k\lambda_i^{\frac{p}{2}}(TT^*)+\sum_{i=1}^k\lambda_i^{\frac{p}{2}}(x^2SS^*)$$
for all $0<x<1,$ and $UTV=\mathrm{diag}(s_1(T),\ldots,s_n(T))$ for some unitary matrices $U,V\in M_n$.
\item[$\mathrm{ (1)}$]If  $s_k(T)=0,$  then $T\perp S$.
\item[$\mathrm{ (2)}$]
If $s_k(T)>0,$ then $ USV=0_{k+\ell}\oplus \hat{S}$ with $\hat{S}\in M_{n-k-\ell}$,
where $\ell$ is the largest integer such that $s_{k+\ell}(T)=s_k(T).$
\end{corollary}
\begin{proof}
If $s_k(T)=0$, then $\lambda_k(T^*T)=\lambda_k(TT^*)=0$. We can use Lemma  \ref{lemma8} twice to conclude that $T^*T\perp S^*S$ and  $TT^*\perp SS^*$, and hence $T\perp S.$

If $s_k(T)>0$, 
 then we have $$V^*T^*TV=\mathrm{diag}(s_1^2(T),\ldots,s_{n}^2(T))\quad
  \hbox{and}\quad  UTT^*U^*=\mathrm{diag}(s_1^2(T),\ldots,s_{n}^2(T)).$$
  Notice that $\lambda_{i}(TT^*)=\lambda_{i}(T^*T)=s^2_i(T)$ for $i=1,\ldots,n$. Thus,  $$\lambda_k(T^*T)=\lambda_k(TT^*)=s_k^2(T)>0$$ and $\ell$ is the largest integer such that $$\lambda_{k+\ell}(TT^*)=\lambda_{k}(TT^*)\quad  \hbox{and} \quad \lambda_{k+\ell}(T^*T)=\lambda_{k}(T^*T).$$ Then we use Lemma \ref{lemma8} twice to conclude that
$$VS^*SV=0_{k+\ell}\oplus C\quad \hbox{and}\quad USS^*U^*=0_{k+\ell}\oplus D.$$
It follows that
$USV=0_{k+\ell}\oplus \hat{S}.$
\end{proof}
The following result  originates from the
 last two paragraphs of the proof of Theorem 2.1 in  \cite{FHLS4}.
\begin{lemma}\label{lemma9}
Let $\phi:M_{mn}\to M_{mn}$ be a linear map.
Suppose for any unitary matrix $X \in M_m$ and  integer  $1\leq i\leq m$, there exists a unitary matrix $W_X$    such that
\begin{equation}\notag
\phi(XE_{ii}X^*\otimes B)=W_X(E_{ii}\otimes
\varphi_{i,X} (B)){W}_X^*\quad \hbox{for all\quad}B\in M_n,
\end{equation}
where $\varphi_{i,X}$ is the identity map or the transposition map and $W_I=I_{mn}$. Then
$$\phi(A\otimes B)=\varphi_{1}(A)\otimes \varphi_2(B)\quad \hbox{for all\quad}A\in M_{m} \hbox{ and } B\in M_n,$$
where $\varphi_1$ is a linear map and $\varphi_2$ is the identity map or the transposition map.
\end{lemma}
\begin{proof}
For any real symmetric matrix $S\in M_n$ and any  unitary matrix $X\in M_m$,
$${\phi}(I_m\otimes S)=\sum_{i=1}^m{\phi}(XE_{ii}X^*\otimes S)={W}_X(I_m\otimes S){W}_X^*.$$
Since $W_I=I_{mn}$, it follows that $$ {W}_X(I_m\otimes S){W}_X^*=W_I(I_m\otimes S){W}_I^*=I_m\otimes S.$$
Thus, ${W}_X$ commutes with $I_m\otimes S$ for all real symmetric $S\in M_n.$ This yields that ${W}_X=Z_X\otimes I_n$ for some unitary matrix $Z_X\in M_m$, and hence \begin{equation}\label{eq25}
{\phi}(XE_{ii}X^*\otimes B)=(Z_XE_{ii}Z_X^*)\otimes \varphi_{i,X}(B)\quad \hbox{for all\quad} i=1,\ldots,m\hbox{ and }B\in M_n.
\end{equation}
Define  linear maps $\hbox{tr}_1 : M_{mn}\to M_{n}$ and  $
\hbox{Tr}_1: M_{mn}\to M_{n}$ as
$$\hbox{tr}_1(A\otimes B)=\hbox{tr} (A)B\quad \hbox{and} \quad \hbox{Tr}_1(A\otimes
 B)=\hbox{tr}_1({\phi}(A\otimes B))$$
for all $A\in M_m$ and $B\in M_n.$ The map $\hbox{tr}_1$ is also called the partial trace function in quantum science.
Then
$$\hbox{Tr}_1(XE_{ii}X^*\otimes B)=\varphi_{i,X}(B).$$  Note that $\hbox{Tr}_1$ is linear and therefore continuous and the set
$$\{XE_{ii}X^*\mid 1\leq i\leq m, X\in M_m \hbox{ is unitary}\}=\{xx^*\in  M_m\mid x^*x=1\}$$
is connected. So, all the maps $\varphi_{i, X}$ are the same and hence we can rewrite (\ref{eq25}) as
$${\phi}(XE_{ii}X^*\otimes B)=(Z_XE_{ii}Z_X^*)\otimes \varphi_{2}(B)\quad \hbox{for all\quad} i=1,\ldots,m\hbox{ and }B\in M_n,$$
where $\varphi_2$ is the identity map or the transposition map.
By the linearity of ${\phi}$, it follows that
$${\phi}(A\otimes B)=\varphi_1(A)\otimes \varphi_{2}(B)\quad \hbox{for all\quad} A\in M_m\hbox{ and }B\in M_n$$ for some linear map $\varphi_1.$
\end{proof}
\par

Now we are ready to present the proof of Theorem 2.1.

\noindent
{\bf Proof of Theorem 2.1.} Notice that the $(p,k)$ norm reduces to the spectral norm when $k=1$. It was shown in  \cite{FHLS4}  that a linear map $\phi$ preserves the spectral norm of tensor products $A\otimes B$ for all $A\in M_m$ and $B\in M_n$ if and only if $\phi$ has form $A\otimes B\mapsto U(\varphi_1(A)\otimes \varphi_2(B))V$ for some unitary matrices $U,V\in M_{mn}$, where $\varphi_s$ is the identity map or the transposition map for $s=1,2$. So we need only consider the case when  $k\geq 2$ in the following discussion.  Since the sufficiency part is clear, we consider only the necessity part.

Suppose a linear map $\phi: M_{mn}\to M_{mn}$ satisfies (\ref{CH3t1}) and $k
\geq 2$. We need the following three claims.

\begin{claim}\label{cl1}
 For any unitary matrices $X\in M_m\hbox{ and } Y\in M_n$,  we have
$$\phi (XE_{ii}X^*\otimes YE_{jj}Y^*)\perp \phi(XE_{ii}X^*\otimes YE_{ss}Y^*)$$
 and
 $$\phi (XE_{jj}X^*\otimes YE_{ii}Y^*)\perp \phi(XE_{ss}X^*\otimes YE_{ii}Y^*)$$
for any possible $i,j,s$ with $j\neq s$. Moreover,
 $$\textnormal{rank}(\phi(XE_{ii}X^*\otimes Y(E_{jj}+E_{ss})Y^*))\leq  k$$ and $$\textnormal{rank}(\phi(X(E_{jj}+E_{ss})X^*\otimes YE_{ii}Y^*))\leq  k, $$
for any possible $i,j,s$ with $j\neq s $.
\end{claim}

\textit{Proof of Claim 1. } For simplicity, we denote $$T=\phi (XE_{ii}X^*\otimes YE_{jj}Y^*)\quad \hbox{and}\quad S=\phi(XE_{ii}X^*\otimes YE_{ss}Y^*).$$  We need to show  $$T\perp S \quad{\rm and }\quad \mathrm{rank}(T+S)\leq k.$$
With the assumption in (\ref{CH3t1}), we have
{$$\|\phi(XFX^{*}\otimes YGY^{*})\|_{(p,k)}=\|XFX^{*}\otimes YGY^{*}\|_{(p,k)}=\|F\otimes G\|_{(p,k)}$$
for all $F\in M_m$ and $G\in M_n$. It follows that $$\|T+xS\|_{(p,k)}^p=\|E_{ii}\otimes (E_{jj}+xE_{ss})\|^{p}_{(p,k)}=1+x^p,$$ 
$$\|T-xS\|_{(p,k)}^p=\|E_{ii}\otimes (E_{jj}-xE_{ss})\|^{p}_{(p,k)}=1+x^p,$$
$$\|T\|^p_{(p,k)}=1 \quad \hbox{and}\quad \|xS\|^p_{(p,k)}=x^p$$ for all $0<x<1$.}
 We can conclude from the above equalities that
\begin{equation}\label{CH3eqq1}
\|T+xS\|_{(p,k)}^p+\|T-xS\|_{(p,k)}^p=2\|T\|_{(p,k)}^p+2\|xS\|_{(p,k)}^p\quad \hbox{for all\quad } 0<x<1.
\end{equation}
Applying Corollary \ref{CH3c1} with $A=T$ and $B=xS$, we get
\begin{equation}\label{CH3eqq2}
\|T+xS\|_{(p,k)}^p+\|T-xS\|_{(p,k)}^p\geq 2\sum_{i=1}^k\lambda_i^{\frac{p}{2}}(T^*T+x^2S^*S)
\end{equation}
for all $0<x<1.$
Since $\|T\|_{(p,k)}^p=\displaystyle\sum_{i=1}^k\lambda_i^{\frac{p}{2}}(T^*T)$ and $\|xS\|_{(p,k)}^p=\displaystyle\sum_{i=1}^k\lambda_i^{\frac{p}{2}}(x^2S^*S),$
it follows from  (\ref{CH3eqq1}) and (\ref{CH3eqq2})  that
\begin{equation}\label{with1}
\sum_{i=1}^k\lambda_i^{\frac{p}{2}}(T^*T+x^2S^*S)\leq \sum_{i=1}^k\lambda_i^{\frac{p}{2}}(T^*T)+\sum_{i=1}^k\lambda_i^{\frac{p}{2}}(x^2S^*S)
\end{equation}
for all $0<x<1.$

Replacing the role of $(T,S)$ with $(T^*,S^*)$ in the above argument, we have
 \begin{equation}\label{with2}
\sum_{i=1}^k\lambda_i^{\frac{p}{2}}(TT^*+x^2SS^*)\leq \sum_{i=1}^k\lambda_i^{\frac{p}{2}}(TT^*)+\sum_{i=1}^k\lambda_i^{\frac{p}{2}}(x^2SS^*)
\end{equation}
for all $0<x<1.$
We claim that $s_k(T)=0$.  Otherwise, suppose
 $s_k(T)>0$. Then by (\ref{with1}) and (\ref{with2}), we can apply Corollary \ref{c2} to conclude that there exist  unitary matrices $U,V\in M_{n}$ such that
 $$UTV=\mathrm{diag}(s_1(T),\ldots,s_{mn}(T))\quad \hbox{and}\quad USV=0_{k+\ell}\oplus \hat{S},$$
 where $\ell$ is the largest integer such that $s_{k+\ell}(T)=s_k(T).$
 Thus,  there exists a sufficiently small number $t>0$
    such that the largest $k$ singular values of  $T+tS$ are $s_1(T),\ldots,s_k(T)$. Since $\|T\|_{(p,k)}^p=\|E_{ii}\otimes E_{jj}\|_{(p,k)}^p=1,$
 we have $$\|T+tS\|_{(p,k)}^p=\sum_{i=1}^ks_i^p(T+tS)=\sum_{i=1}^ks_i^p(T)=\|T\|_{(p,k)}^p=1,$$ which contradicts  the fact that $$\|T+xS\|_{(p,k)}^p=\|E_{ii}\otimes (E_{jj}+xE_{ss})\|_{(p,k)}^p=1+x^p\quad \hbox{for all\quad} 0<x<1.$$  So, our claim is correct, that is, $s_k(T)=0.$

 Now applying Corollary \ref{c2} again, we have $T\perp S.$ 	Notice that $$\|T+S\|_{(p,k)}^p=\|T\|^p_{(p,k)}+\|S\|_{(p,k)}^p.$$	Applying Lemma
 \ref{CH3le2} on $S$ and $T$ we have
 $$\mathrm{rank}(T+S)\leq k.$$   
 Similarly, we can also show that $$\phi (XE_{jj}X^*\otimes YE_{ii}Y^*)\perp \phi(XE_{ss}X^*\otimes YE_{ii}Y^*)$$
 and $$\textnormal{rank}(\phi(X(E_{jj}+E_{ss})X^*\otimes YE_{ii}Y^*))\leq  k $$
for any possible $i,j,s$ with $j\neq s$.

\begin{claim}\label{cl2}
For any unitary matrices $X\in M_m\hbox{ and } Y\in M_n$, we have
$$\phi (XE_{ii}X^*\otimes Y(E_{jj}+E_{ss})Y^*)\perp \phi(XE_{tt}X^*\otimes Y(E_{jj}+ E_{ss})Y^*)$$
whenever $i\neq t.$
\end{claim}

\textit{Proof of Claim 2. }  For simplicity, we denote
$$T=\phi (XE_{ii}X^*\otimes Y(E_{jj}+E_{ss})Y^*) \quad {\rm and }\quad S=\phi(XE_{tt}X^*\otimes Y(E_{jj} +E_{ss})Y^*).$$ We need to show   $S\perp T.$
Applying Corollary \ref{CH3c1} on $T$ and $xS$, we get \begin{equation}\label{CH3e12}
\|T+xS\|_{(p,k)}^p+\|T-xS\|_{(p,k)}^p\geq 2\sum_{i=1}^k\lambda_i^{\frac{p}{2}}(T^*T+x^2S^*S)
\end{equation}
for all $0<x<1.$
With the assumption in (\ref{CH3t1}),
we have
\begin{itemize}
\item[(\romannumeral1)] $\|T+xS\|_{(p,k)}^p+\|T-xS\|_{(p,k)}^p=2\|T\|_{(p,k)}^p+2\|xS\|_{(p,k)}^p$ for the case $k\geq 4$;
\item[(\romannumeral2)]$\|T+xS\|_{(p,k)}^p+\|T-xS\|_{(p,k)}^p=2\|T\|_{(p,k)}^p+\|xS\|_{(p,k)}^p$ for the case $k=3$;
\item[(\romannumeral3)]$\|T+xS\|_{(p,k)}^p+\|T-xS\|_{(p,k)}^p=2\|T\|_{(p,k)}^p$ for the case $k=2.$
\end{itemize}
So we can conclude that for any integer $k\geq 2$,
\begin{equation}
\begin{split}
\|T+xS\|_{(p,k)}^p+\|T-xS\|_{(p,k)}^p&\leq  2\|T\|_{(p,k)}^p+2\|xS\|_{(p,k)}^p\\
&=2\sum_{i=1}^k\lambda_i^{\frac{p}{2}}(T^*T)+2\sum_{i=1}^k\lambda_i^{\frac{p}{2}}(x^2S^*S).
\end{split}
\end{equation}
 It follows that
\begin{equation}\label{CH3eq20}
\sum_{i=1}^k\lambda_i^{\frac{p}{2}}(T^*T+x^2S^*S)\leq \sum_{i=1}^k\lambda_i^{\frac{p}{2}}(T^*T)+\sum_{i=1}^k\lambda_i^{\frac{p}{2}}(x^2S^*S)
\end{equation}
for all $0<x<1.$
The above observations also hold if $(T,S)$ is replaced by $(T^*,S^*)$, that is,
 \begin{equation}\label{CH3eq21}
\sum_{i=1}^k\lambda_i^{\frac{p}{2}}(TT^*+x^2SS^*)\leq \sum_{i=1}^k\lambda_i^{\frac{p}{2}}(TT^*)+\sum_{i=1}^k\lambda_i^{\frac{p}{2}}(x^2SS^*)
\end{equation}
for all $0<x<1.$

If $s_k(T)=0$, then applying Corollary \ref{c2}  we have $T\perp S$. Otherwise, $s_k(T)>0$.
 Notice that  Claim \ref{cl1} implies $\mathrm{rank}(T)\leq k.$ Thus, by (\ref{CH3eq20}) and (\ref{CH3eq21}), we can apply Corollary \ref{c2} to conclude that  there exist  unitary matrices $U,V\in M_{mn}$ such that  $$UTV=\mathrm{diag}(s_1(T),\ldots,s_{k}(T))\oplus 0_{mn-k}\quad \hbox{and}\quad USV=0_{k}\oplus \hat{S}.$$
It follows that $T\perp S.$ This completes the proof.

\begin{claim}\label{cl3}
 For any unitary matrices  $X\in M_m$ and $Y\in M_n$,
 $$\phi(XE_{ii}X^*\otimes YE_{jj}Y^*)\perp \phi(XE_{rr}X^*\otimes YE_{ss}Y^*)\quad \hbox{for any }(i,j)\neq (r,s). $$
\end{claim}

\textit{Proof of Claim 3. }
If $i=r$ or $j=s$,  then  applying Claim \ref{cl1} directly we have
$$\phi(XE_{ii}X^*\otimes Y E_{jj}Y^*)\perp \phi (XE_{rr}X^*\otimes YE_{ss}Y^*).$$
Next, we suppose that  $i\neq r$ and $j\neq s$.
With Claim \ref{cl1},  we have
\begin{equation}\label{as31}
\phi(XE_{ii}X^*\otimes Y E_{jj}Y^*)\perp \phi (XE_{ii}X^*\otimes YE_{ss}Y^*)
\end{equation}
and
\begin{equation}\label{as32}
\phi(XE_{rr}X^*\otimes Y E_{jj}Y^*)\perp \phi (XE_{rr}X^*\otimes YE_{ss}Y^*).
\end{equation}
With Claim \ref{cl2}, we have
\begin{equation}\label{as33}
\phi(XE_{ii}X^*\otimes Y (E_{jj}+E_{ss})Y^*)\perp \phi (XE_{rr}X^*\otimes Y(E_{jj}+E_{ss})Y^*).
\end{equation}
Applying Lemma \ref{adle2},  we conclude from (\ref{as31}) and (\ref{as33}) that
\begin{equation}\label{as34}
\phi(XE_{ii}X^*\otimes Y E_{jj}Y^*)\perp \phi (XE_{rr}X^*\otimes Y(E_{jj}+E_{ss})Y^*).
\end{equation}
Then applying Lemma \ref{adle2} again,  we can conclude from (\ref{as32}) and (\ref{as34}) that
$$\phi(XE_{ii}X^*\otimes Y E_{jj}Y^*)\perp \phi (XE_{rr}X^*\otimes YE_{ss}Y^*).$$

 Now we prove that $\phi$ has the desired form (\ref{CH3t2}).
For any unitary matrix $Y\in M_n$, applying Claim \ref{cl1}  and Claim \ref{cl3}   we know
$$\mathscr{F}=\{\phi(E_{ii}\otimes YE_{jj}Y^*):i=1,\ldots,m\hbox{ and }j=1\ldots,n\}$$
is a set of $mn$ orthogonal matrices in $M_{mn}$.
By Claim \ref{cl1},   each matrix in $\mathscr{F}$ has   exactly one nonzero singular value, which equals 1. Thus,
there exist  unitary matrices $U_Y,V_Y\in M_{mn}$  such that
\begin{equation}\label{ch2eq6}
\phi(E_{ii}\otimes YE_{jj}Y^*)=U_Y(E_{ii}\otimes E_{jj})V_Y^*
\end{equation}
for all $ i=1,\ldots, m \hbox{ and } j=1,\ldots, n.$
Without loss of generality, we may assume that $U_I=V_I=I_{mn}$, i.e.,
\begin{equation}\label{ch2eq7}
\phi(E_{ii}\otimes E_{jj})=E_{ii}\otimes E_{jj}
\end{equation}
for all $i=1,\ldots, m \hbox{ and } j=1,\ldots, n.$
By (\ref{ch2eq6}) and (\ref{ch2eq7}), we have
\begin{itemize}
\item[(\romannumeral1)] $I_{mn}=\phi(I_m\otimes I_n)=U_Y(I_m\otimes I_n)V_Y^*$;
\item[(\romannumeral2)] $E_{ii}\otimes I_n=\phi(E_{ii}\otimes I_n)=U_Y(E_{ii}\otimes I_n)V_Y^*  \hbox{ for all } i=1,\ldots,m.$
\end{itemize}
It follows that  $U_Y=V_Y$ and   $U_Y$ commutes with $E_{ii}\otimes I_n$ for all $i=1,\ldots, m$.
Therefore,  $U_Y$ commutes with $E_{11}\otimes I_n+2E_{22}\otimes I_n+\cdots+mE_{mm}\otimes I_n,$  which implies that $U_Y=\bigoplus\limits_{i=1}^{m}U_{i,Y}$ with unitary matrices $U_{i,Y}\in M_n$.  It follows that
$$\phi(E_{ii}\otimes YE_{jj}Y^*)=E_{ii}\otimes U_{i,Y}E_{jj}U_{i,Y}^*.$$
So far, we have showed that for any unitary  matrix $Y\in M_n$, there exists a unitary matrix  $U_{i,Y}\in M_n$ depending on $i$ and $Y$ such that
$$\phi(E_{ii}\otimes YE_{jj}Y^*)=E_{ii}\otimes U_{i,Y}E_{jj}U_{i,Y}^*\quad \hbox{for\quad}j=1,\ldots,n.$$
{By the linearity of $\phi$,}
 we conclude from the above equation that for any $i=1,\ldots, m$, there exists a linear map
 $\psi_i$ such that
$$\phi(E_{ii}\otimes B)=E_{ii}\otimes \psi_{i}( B) \quad \hbox{for all\quad} B \in M_n.$$
  Let $\hat{k}=\min\{k,n\}$. Then it is easy to check that
$$\|\psi_i(B)\|_{(p,\hat{k})}=\|E_{ii}\otimes \psi_i(B)\|_{(p,k)}=\|E_{ii}\otimes B\|_{(p,k)}=\|B\|_{(p,\hat{k})}$$ for all $B\in M_n$. That is, $\psi_i$ is a linear map on $M_n$  preserving the $(p,\hat{k})$ norm.
Thus, by Theorem 1 in \cite{LT1988},
$\psi_i$ has  form $B\mapsto W_i B\widetilde{W}_i$ or $B\mapsto W_i B^T\widetilde{W}_i$ for some unitary matrices
$W_i, \widetilde{W}_i\in M_n.$
Let $W=\bigoplus\limits_{i=1}^m W_i$
and $\widetilde{W}=\bigoplus\limits_{i=1}^m\widetilde{W}_i$. It follows  that for any
$i=1,\ldots, m,$
$$\phi(E_{ii}\otimes B)=W(E_{ii}\otimes
\varphi_i (B))\widetilde{W}\quad \hbox{for all\quad}B\in M_n,$$
where $\varphi_i$ is the identity map or the transposition map.
Recall  that $I_{mn}=\phi(I_m\otimes I_n). $ Thus, we have $\widetilde{W} = W^*$.

Applying  Claim \ref{cl3}  again, we can repeat the  same argument  above to  show that for any unitary matrix $X \in M_m$ and any integer $1\leq i\leq m$, there exists a unitary matrix $W_X$ such that
\begin{equation}\notag
\phi(XE_{ii}X^*\otimes B)=W_X(E_{ii}\otimes
\varphi_{i,X} (B)){W}_X^*\quad \hbox{for all\quad}B\in M_n,
\end{equation}
where $\varphi_{i,X}$ is the identity map or the transposition map. We may further assume that $W_I=I_{mn}$. Then applying Lemma \ref{lemma9}, we have
$$\phi(A\otimes B)=\varphi_1(A)\otimes \varphi_2(B) \quad \hbox{for all\quad}A\in M_m\hbox{ and }B\in M_n,$$
where  $\varphi_2$ is the identity map or the transposition map and  $\varphi_1$  is a linear map on $M_m$. Let $\widetilde{k}=\min\{k,m\}.$ It is easy to verify that $\varphi_1$ is a linear map on $M_m$ preserving the $(p,\widetilde{k})$ norm. Hence, $\varphi_1$ also has the form $A\mapsto UAV$ or $A\mapsto UA^TV$ for some unitary  matrices $U, V\in M_m$. This completes the proof.\qed

\section{Multipartite system}

In this section we extend Theorem \ref{CH3theo1} to multipartite systems. The proof of the following lemma can be found in the proof of Theorem 3.1 in \cite{FHLS4}. For completeness, we present it as follows.
\begin{lemma}\label{lemma10}
Given  an integer $m\geq 2$,
let  $n_i \geq 2 $ be integers for  $i=1,\ldots,m$ and  $N=\prod\limits_{i=1}^{m}n_i$.
{Let   $\phi:M_N\to M_N$ be a linear map. Suppose  for any unitary matrices $X_i\in M_{n_i}$ and any { integers $ 1\leq j_i\leq n_i$ with  $ 1\leq i\leq  m-1$,} there exists a  unitary matrix $W_X\in M_N$ } depending on $X=(X_1,\ldots,X_{m-1})$ such that
\begin{equation}\label{ch}
\phi\left (\bigotimes\limits_{i=1}^{m-1}X_iE_{j_ij_i}
X_i^*\otimes B\right )=
W_{X}\left (\bigotimes\limits_{i=1}^{m-1}E_{j_ij_i}\otimes\varphi_{j_1,\ldots,j_{m-1},X}(B)\right )W_{X}^*
\end{equation}
for all $B\in M_{n_m},$
where $\varphi_{j_1,\ldots,j_{m-1},X}$ is the
 identity map or the transposition map and $W_X=I_N$ when $X=(I_{n_1},\ldots,I_{n_{m-1}})$.
 Then $$\phi(A_1\otimes \cdots\otimes  A_{m-1}\otimes B)=\varphi_1(A_1\otimes \cdots \otimes A_{m-1})\otimes \varphi_2(B),$$
 where $\varphi_1$ is a linear map and $\varphi_2$ is the identity map or the transposition map.
\end{lemma}
\begin{proof}
Considering all symmetric real matrices as   in the  proof of Lemma \ref{lemma9}, one can conclude that there exists some unitary matrix $Z_{X}$ such that\begin{equation}\label{eq26}
\phi\left (\bigotimes\limits_{i=1}^{m-1}X_iE_{j_ij_i}X_i^*\otimes B\right )= \left(Z_{X}\left (\bigotimes\limits_{i=1}^{m-1}E_{j_ij_i}\right )Z_{X}^*\right) \otimes \varphi_{j_1,\ldots,j_{m-1},X}(B)
\end{equation}
for all $B\in M_{n_m}$ and integers $1\leq j_i\leq  n_i $ with  $ 1\leq i\leq  m-1$. Define linear maps $\mathrm{tr_1}:M_N\to M_{n_m}$ and $\mathrm{Tr_1}:M_N\to M_{n_m}$ by $$\mathrm{tr_1}(A\otimes B)=\mathrm{tr}(A)B\quad \hbox{and} \quad \mathrm{Tr_1}(A\otimes B)=\mathrm{tr_1}(\phi(A\otimes B ))$$
for all $A\in M_{n_1\cdots n_{m-1}}$ and $B\in M_{n_m}$.
Then
$$\mathrm{Tr_1}\left(\bigotimes\limits_{i=1}^{m-1}X_iE_{j_ij_i}X_i^*\otimes B\right )=\varphi_{j_1,\ldots,j_{m-1},X}(B).$$
 Notice that $\hbox{Tr}_1$ is  linear and therefore continuous. Besides,  the set
 \begin{multline}
\left\{\bigotimes\limits_{i=1}^{m-1}X_iE_{j_ij_i}X_i^*\mid 1\leq j_i\leq n_i \hbox{ and } X_i\in M_{n_i} \hbox{ is unitary for }i=1,\ldots,m-1\right\}\\
=\left\{\bigotimes\limits_{i=1}^{m-1}x_ix_i^*\mid x_i\in \mathbb{C}^{n_i} \hbox{ with } x_i^*x_i=1 \hbox{ for }i=1,\ldots,m-1\right\}
\end{multline}
is connected. So, all the maps $\varphi_{j_1,\ldots,j_{m-1},X}$  are the same. Then (\ref{eq26}) can be rewritten as
$$\phi\left (\bigotimes\limits_{i=1}^{m-1}X_iE_{j_ij_i}X_i^*\otimes B\right )= \left(Z_{X}\left (\bigotimes\limits_{i=1}^{m-1}E_{j_ij_i}\right )Z_{X}^*\right) \otimes \varphi_{2}(B),$$
where $\varphi_2$
 is the identity map or the transposition map. With  the linearity of $\phi$, it follows that
 $$\phi(A_1\otimes \cdots \otimes A_{m-1}\otimes B)=\varphi_1(A_1\otimes \cdots \otimes A_{m-1})\otimes \varphi_2(B)$$
 for some linear map $\varphi_1.$
 \end{proof}

\par

\begin{theorem}\label{CH3theo2}
Given  an integer $m\geq 2$,
let  $n_i \geq 2 $ be integers for  $i=1,\ldots,m$ and  $N=\prod\limits_{i=1}^{m}n_i.$ Then for any real number $p>2$ and any positive integer $k \leq N,$
a linear map $\phi: M_N \to M_N$ satisfies
\begin{equation}\label{CH3th31}
\|\phi(A_1\otimes\cdots\otimes A_m)\|_{(p,k)}=\|A_1\otimes\cdots\otimes A_m\|_{(p,k)}
\end{equation}
for all  $A_i\in M_{n_i}, i=1,\ldots,m$,
if and only if there exist unitary matrices $U,V \in M_N$ such that \begin{equation}\label{CH3form}
\phi(A_1\otimes\cdots\otimes A_m)=U(\varphi_1(A_1)\otimes\cdots\otimes \varphi_m(A_m))V
\end{equation}
for all $A_i\in M_{n_i},i=1,\ldots,m,$
where $\varphi_i$ is the identity map or  the transposition map $A\mapsto A^T $ for $i=1,\ldots,m$.
\end{theorem}
\begin{proof}
By Theorem 3.2 of \cite{FHLS4}, we know the result holds for $k=1$. So we may assume $k\ge 2$.

We use induction on $m$. By Theorem \ref{CH3theo1},   the result holds for $m=2$. Now suppose that $m\geq 3$ and   the result   holds for any $(m-1)$-partite system. We need to show that the result holds for any $m$-partite system.

We first show that
for  any unitary matrices $X_i\in M_{n_i},$ $i=1,\ldots,m,$
\begin{equation}\label{CH3claim}
\phi(X_1E_{i_1i_1}X_1^*\otimes\cdots \otimes X_mE_{i_mi_m}X_m^*)\perp\phi(X_1E_{j_1j_1}X_1^*\otimes \cdots\otimes X_mE_{j_mj_m}X_m^*)
\end{equation}
for any  $(i_1,\ldots,i_m)\neq (j_1,\ldots,j_m).$
Without loss of generality, we need only to prove that (\ref{CH3claim}) holds when  $X_i=I_{n_i}$  for  $i=1\ldots,m$.  By Lemma \ref{adle2}, it  suffices to show that for any integer $1\leq s\leq m,$ we have
\begin{multline}\label{CH3eq33}
\phi\left (\bigotimes\limits_{u=1}^{s-1}(E_{i_ui_u}+E_{j_uj_u})\otimes E_{i_si_s}\otimes \bigotimes\limits_{u=s+1}^mE_{i_ui_u}\right )\\
\perp \phi\left (\bigotimes\limits_{u=1}^{s-1}(E_{i_ui_u}+E_{j_uj_u})\otimes E_{j_sj_s}\otimes \bigotimes\limits_{u=s+1}^mE_{i_ui_u}\right )
\end{multline}
for all $\mathrm{i}=(i_1,\ldots,i_m)$ and $\mathrm{j}=(j_1,\ldots,j_m)$ with $i_u\neq j_u$ for $ u=1,\ldots, s.$ Denote by $A_s(\mathrm{i,j})$ and $B_s(\mathrm{i,j})$ the  two matrices in (\ref{CH3eq33}) accordingly.  It is easy to check that for any integer $1\leq s\leq m$ and real number $0<x<1$,
\begin{multline}\notag
\|A_s(\mathrm{i,j})+xB_s(\mathrm{i,j})\|_{(p,k)}^p+\|A_s(\mathrm{i,j})
-xB_s(\mathrm{i,j})\|_{(p,k)}^p\\
\leq 2\|A_s(\mathrm{i,j})\|^p_{(p,k)}+2\|xB_s(\mathrm{i,j})\|^p_{(p,k)}.\notag
\end{multline}
Then applying the same argument as in the proof of Theorem \ref{CH3theo1}, we conclude that  for any integer $1\leq s \leq m$ and real number $0<x<1$,  
 \begin{multline}\label{CH3eq411}
\sum_{i=1}^k\lambda_i^{\frac{p}{2}}\big(A_s^*(\mathrm{i,j})A_s(\mathrm{i,j}\big)+x^2B_s^*(\mathrm{i,j})B_s(\mathrm{i,j})\big)\\
\leq \sum_{i=1}^k\lambda_i^{\frac{p}{2}}\big(A^*_s(\mathrm{i,j})A_s(\mathrm{i,j})\big)+\sum_{i=1}^k\lambda_i^{\frac{p}{2}}\big(x^2B_s^*(\mathrm{i,j})B_s(\mathrm{i,j})\big)
\end{multline}
and
 \begin{multline}\label{CH3eq41}
\sum_{i=1}^k\lambda_i^{\frac{p}{2}}\big(A_s(\mathrm{i,j})A^*_s(\mathrm{i,j})+x^2B_s(\mathrm{i,j})B_s^*(\mathrm{i,j})\big) \\\leq \sum_{i=1}^k\lambda_i^{\frac{p}{2}}\big(A_s(\mathrm{i,j})A^*_s(\mathrm{i,j})\big)+\sum_{i=1}^k\lambda_i^{\frac{p}{2}}\big(x^2B_s(\mathrm{i,j})B^*_s(\mathrm{i,j})\big)
\end{multline}
for all  $\mathrm{i}=(i_1,\ldots,i_m)$ and $\mathrm{j}=(j_1,\ldots,j_m)$ with $i_u\neq j_u$ for $ u=1,\ldots, s.$ Now we distinguish two cases.\\

\noindent\textit{Case 1.} Suppose  $k>2^{m-1}$.  Then
 \begin{equation}\label{CH3eq34}
\|A_s(\mathrm{i,j})+xB_s(\mathrm{i,j})\|_{(p,k)}^p=2^{s-1}+ a_s x^p\quad \hbox{ for\quad }0<x<1,
\end{equation}
where $a_s=2^{s-1}$ for $s=1,\ldots,m-1$ and $a_m=\min\{k-2^{m-1},2^{m-1}\}$.
 We claim  that
 \begin{equation}\label{eqh714}
 s_k(A_{s}(\mathrm{i,j}))=0 {\rm \quad for \quad}  s=1,\ldots,m.
  \end{equation} Otherwise, $s_k(A_s(\mathrm{i,j}))>0$ for some $1\leq s\leq m$.
 Then by (\ref{CH3eq411}) and  (\ref{CH3eq41}), we can use the same argument as in Claim \ref{cl1} in the proof of Theorem \ref{CH3theo1} to conclude that there exists a sufficiently small $x>0$ such that
 $$\|A_s(\mathrm{i,j})+xB_s(\mathrm{i,j})\|_{(p,k)}^p=\|A_s(\mathrm{i,j})\|_{(p,k)}^p=2^{s-1},$$ which  contradicts (\ref{CH3eq34}). Thus,  (\ref{eqh714}) holds. Applying Corollary \ref{c2},  we have $$A_s(\mathrm{i,j})\perp B_s(\mathrm{i,j}) {\rm\quad for \quad}
  s=1,\ldots,m.$$

\noindent\textit{Case 2.}
Suppose  $k\leq 2^{m-1}$.
Let $s_0$ be the integer such that $2^{s_0-1}< k\leq 2^{s_0}.$ We can use the same argument as in \textit{Case 1} to show that for any integer $1\leq s\leq s_0$, \begin{equation}\label{CH3eq36}
A_{s}(\mathrm{i,j})\perp B_s(\mathrm{i,j})\quad\hbox{ and}\quad s_k(A_s(\mathrm{i,j}))=0
\end{equation} for all $\mathrm{i}=(i_1,\ldots,i_m)$  and $\mathrm{j}=(j_1,\ldots,j_m)$ with $i_u\neq j_u$ for  $u=1,\ldots,s.$

Next, we use induction on $s$ to prove that  for
 any  $s_0+1\leq s \leq m,$  $\mathrm{i}=(i_1,\ldots,i_m
)$  and $\mathrm{j}=(j_1,\ldots,j_m)$ with $i_u\neq j_u$ for $u=1,\ldots, s,$
 there exist unitary matrices $U,V\in M_N$
  depending on $s$ and $(\mathrm{i,j})$ such that
\begin{equation}\label{CH3cal}
UA_{s}(\mathrm{i,j})V=I_{2^{s-1}}\oplus 0_{N-2^{s-1}}\quad \hbox{and}\quad
A_{s}(\mathrm{i,j})\perp B_s(\mathrm{i,j}).
\end{equation}
First
by (\ref{CH3eq36}), we have $A_{s_0}(\mathrm{i,j})\perp B_{s_0}(\mathrm{i,j})$ and there exist unitary matrices $U,V\in M_N$  and an integer $0\leq r< k$ such that
$$UA_{s_0}(\mathrm{i,j})V=\mathrm{diag}(a_1,\ldots,a_r)\oplus 0$$
and  $$UB_{s_0}(\mathrm{i,j})V=0_r\oplus \mathrm{diag}(b_{r+1},\ldots,b_N),$$
with $a_1\geq \cdots \geq a_r>0$ and $b_{r+1}\geq \cdots\geq b_N\geq 0$. If $a_1>1$, then by (\ref{CH3eq36}),   applying Lemma \ref{adle2} we have $s_1\Big(\phi\Big(\bigotimes\limits_{u=1}^mE_{i_ui_u}\Big)\Big)>1$ for some $(i_1,\ldots,i_m)$. It follows that $\Big\|\phi\Big(\bigotimes\limits_{u=1}^mE_{i_ui_u}\Big)\Big\|_{(p,k)}>1$, which contradicts
(\ref{CH3th31}). Thus, $a_1\leq 1$, and similarly $b_{r+1}\leq 1.$
Then we have \begin{equation}\label{CH3eq42}
\sum_{j=1}^ra_j^p+\sum_{j=r+1}^kb_j^p\leq k.
\end{equation}
 Clearly, $a_1\geq \cdots\geq a_r\geq xb_{r+1}\geq \cdots \geq xb_k$ are the largest $k$ singular values of $A_{s_0}(\mathrm{i,j})+xB_{s_0}(\mathrm{i,j})$ for all $0<x\leq \frac{a_r}{b_{r+1}}$. Hence, 
$$\|A_{s_0}(\mathrm{i,j})+xB_{s_0}(\mathrm{i,j})\|^p_{(p,k)}=\sum_{j=1}^ra_j^p+x^p\sum_{j=r+1}^kb_j^p\quad \hbox{for all\quad}0< x\leq \frac{a_r}{b_{r+1}}.$$
On the other hand, with (\ref{CH3th31}), we have
$$\|A_{s_0}(\mathrm{i,j})+xB_{s_0}(\mathrm{i,j})\|^p_{(p,k)}=2^{s_0-1}+x^p(k-2^{s_0-1})\quad \hbox{for all\quad}0<x\leq 1.$$
It follows from the above two equations that
$$\displaystyle\sum_{j=1}^ra_j^p=2^{s_0-1} \quad{\rm and \quad} \displaystyle\sum_{j=r+1}^kb_j^p=k-2^{s_0-1}.$$ Therefore, $ \sum_{j=1}^ra_j^p+ \sum_{j=r+1}^kb_j^p=k,$ i.e., the equality in (\ref{CH3eq42}) holds, which implies that $a_j=1$ for
$j=1,\ldots,r$. Notice that $\|A_{s_0}(\mathrm{i,j})\|_{(p,k)}^p=2^{s_0-1}$. Thus, $r=2^{s_0-1},$ that is, $UA_{s_0}(\mathrm{i,j})V=I_{2^{s_0-1}}\oplus 0_{N-2^{s_0-1}}$. Hence,  (\ref{CH3cal}) holds for $s_0$.

 Suppose that (\ref{CH3cal}) holds for  $s-1$ with $s_0< s\leq m.$ We will show that (\ref{CH3cal}) also holds for $s$.
Notice that
 $A_{s}(\mathrm{i,j})=A_{s-1}(\mathrm{\hat{i},\hat{j}})+B_{s-1}(\mathrm{\hat{i},\hat{j}})$ for some $\mathrm{\hat{i}}=(\hat{i}_1,\ldots,\hat{i}_m)$  and $\mathrm{\hat{j}}=(\hat{j}_1,\ldots,\hat{j}_m)$ with $\hat{i}_u\neq \hat{j}_u$ for $u=1,\ldots, s-1.$
By our assumption, we have
$$UA_{s-1}(\mathrm{\hat{i},\hat{j}})V=I_{2^{s-2}}\oplus 0_{N-2^{s-2}}\quad\hbox{and}\quad UB_{s-1}(\mathrm{\hat{i},\hat{j}})V=0_{2^{s-2}}\oplus I_{2^{s-2}}\oplus 0_{N-2^{s-1}}$$ for some unitary
matrices $U,V\in M_N$. It follows that
$$UA_{s}(\mathrm{i,j})V=I_{2^{s-1}}\oplus 0_{N-2^{s-1}}.$$ Then with (\ref{CH3eq411}) and (\ref{CH3eq41}), we apply Corollary  \ref{c2}   to conclude that
$$UB_s(\mathrm{i,j})V=0_{2^{s-1}}\oplus \hat{B}$$
for some $\hat{B}\in M_{N-2^{s-1}}$.
It follows that $A_s(\mathrm{i,j})\perp B_s(\mathrm{i,j}).$ Now we have proved that (\ref{CH3cal}) holds for $s$. Then we can conclude from the above discussion that for any $s=1,\ldots,m$,
$$A_s(\mathrm{i,j})\perp B_s(\mathrm{i,j})$$ for all   $\mathrm{i}=(i_1,\ldots,i_m)$  and $\mathrm{j}=(j_1,\ldots,j_m)$ with $i_u\neq j_u$ for $ u=1,\ldots, s,$  that is, (\ref{CH3claim}) holds.

It follows that for any unitary matrix $X_m\in M_{n_m}$, there exist unitary matrices $U_{X_m}$ and $V_{X_m}$  such that
\begin{equation}\label{ch241}
\phi\left(\bigotimes\limits_{i=1}^{m-1} E_{j_ij_i}\otimes X_mE_{j_{m}j_{m}}X_m^*\right)=
 U_{X_m}(E_{j_1j_1}\otimes \cdots \otimes E_{j_mj_m})V_{X_m}^*
 \end{equation}
 for all $ j_i=1,\ldots, n_i$ with $1 \leq i\leq m.$ For simplicity, we may assume that $U_{I}=V_{I}=I$, i.e.,
\begin{equation}\label{ch242}
\phi(E_{j_1j_1}\otimes\cdots\otimes E_{j_{m}j_{m}})=
E_{j_1j_1}\otimes \cdots \otimes E_{j_mj_m}.
\end{equation}
 Then we have $\phi(I_N)=I_N$.
Applying a similar argument as in the last two paragraphs of the proof of Theorem \ref{CH3theo1} one can conclude from (\ref{ch241}) and (\ref{ch242}) that  there are unitary matrices $W,\widetilde{W}\in M_N$ such that for all  $  j_i=1,\ldots,n_i$ with $1\leq i\leq m-1$,
$$\phi\left(\bigotimes\limits_{i=1}^{m-1}E_{j_ij_i}\otimes B\right)=
W\left(\bigotimes\limits_{i=1}^{m-1}E_{j_ij_i}\otimes \varphi_{j_1,\ldots,j_{m-1}}(B)\right)\widetilde{W},$$
where $\varphi_{j_1,\ldots,j_{m-1}}$ is the identity map or the transposition map. It follows that  $\phi(I_N)=W\widetilde{W}$. Recall that $\phi(I_N)=I_N$ and $W$ and $\widetilde{W}$ are both unitary matrices.
We have $\widetilde{W}=W^*.$

Following a similar argument as above, one can show that for any $X=(X_1,\ldots,X_{m-1})$ 
and any integer  $  j_i=1\ldots,  n_i$ with $1\leq i\leq m-1$, there exists a  unitary matrix $W_X\in M_N$ such that
\begin{equation}\label{ch243}
\phi\left (\bigotimes\limits_{i=1}^{m-1}X_iE_{j_ij_i}X_i^*\otimes B\right )=W_X\left (\bigotimes\limits_{i=1}^{m-1}E_{j_ij_i}\otimes\varphi_{j_1,\ldots,j_{m-1},X}(B)\right )W_X^*
\end{equation}
for all $B\in M_{n_m},$
where $\varphi_{j_1,\ldots,j_{m-1},X}$ is the
 identity map or transposition map. Denote $\mathrm{I}=(I_{n_1},\ldots,I_{n_{m-1}}).$
For simplicity, we may further assume that
 $W_{\mathrm{I}}=I_N$, i.e., for  any integer $j_i=1,\ldots,  n_i$ with $1\leq i\leq m-1$,
\begin{equation}\notag
\phi\left (\bigotimes\limits_{i=1}^{m-1}E_{j_ij_i}\otimes B\right )=\bigotimes\limits_{i=1}^{m-1}E_{j_ij_i}\otimes\varphi_{j_1,\ldots,j_{m-1},\mathrm{I}}(B)\quad \hbox{for all\quad} B\in M_{n_m}.
\end{equation}
Then we apply Lemma \ref{lemma10} to conclude that
$$\phi(A_1\otimes \cdots\otimes A_{m-1}\otimes B)=\psi(A_1\otimes \cdots\otimes A_{m-1})\otimes \varphi_m(B)$$
for all $B\in M_{n_m}$  and  $A_i\in M_{n_i}$, $1\leq i\leq m-1$, where $\varphi_m$ is the identity map or the transposition map and $\psi$ is a linear map.
Let $\hat{k}=\min\Big\{k,\prod\limits_{i=1}^{m-1}n_i\Big\}.$
 It is easy to check that  $$\|\psi(A_1\otimes\cdots \otimes A_{m-1})\|_{(p,\hat{k})}=\|A_1\otimes\cdots \otimes A_{m-1}\|_{(p,\hat{k})}$$
 for all 
 $A_i \in M_{n_i}, i=1,\ldots,m-1.$
Then by the induction hypothesis, we conclude that there exist unitary matrices $\widetilde{U}$ and 
$\widetilde{V} $ such that
$$\psi(A_1\otimes\cdots \otimes A_{m-1})=\widetilde{U}(\varphi_1(A_1)\otimes \cdots\otimes \varphi_{m-1}(A_{m-1}))\widetilde{V},$$
where $\varphi_i$ is the identity map or the transposition map for $i=1,\ldots,m-1.$
 Therefore $\phi$ has the desired  form and the proof is  completed.
\end{proof}

\section{Conclusion and remarks}

In this paper, we determined the structures of linear maps  on $M_{mn}$  that preserve the $(p,k)$-norms of tensor products of matrices for $p>2$ and $1\leq k\leq mn$. Our study generalized the results in  \cite{FHLS4}
concerning the maps preserving the Ky Fan $k$-norm and and Schatten $p$-norms.
Furthermore, we have also extended the result on bipartite systems to multipartite systems using mathematical induction.

The proofs of the main results in  \cite{FHLS4} heavily rely on Lemmas 2.2 and 2.7 from the same paper, which specifically address Ky Fan $k$-norms and Schatten $p$-norms.  However,  it is important to note that these two lemmas are not applicable for $(p,k)$-norms.  To overcome this limitation, we derived two inequalities involving the eigenvalues of Hermitian matrices and $(p,k)$-norms,  which are presented in Lemma \ref{CH3lemma3} and its corollary. These two inequalities play a crucial role in our proofs.

To characterize linear maps that preserve the $(p, k)$-norms of tensor products of matrices for $1< p\leq 2$,  we attempted to derive an analogue to the inequality (\ref{le2.7}). However, as demonstrated in Remark \ref{remark},  
this analogous  inequality does not hold in general. Consequently the techniques employed in this paper are not able to address the case  $1<p\leq 2$. Therefore, novel approaches need to be introduced to tackle this particular scenario.

\end{document}